\documentclass[12pt,reqno]{amsart}
\usepackage{fullpage}
\usepackage{amssymb}
\usepackage{amsfonts}
\usepackage{amsthm}
\usepackage{mathrsfs}
\usepackage{amsmath}
\usepackage{hyperref}
\usepackage{mathtools}
\mathtoolsset{showonlyrefs}
\usepackage{url}
\usepackage{xcolor}
\usepackage{dsfont}
\usepackage{pgfplots}
\usepackage{comment}
\usepackage{dsfont}
\usepackage{enumitem}
\newtheorem{thm}{Theorem}[section]
\newtheorem{lem}[thm]{Lemma}  
\newtheorem{prop}[thm]{Proposition}
\newtheorem{cor}[thm]{Corollary}

\newtheorem{remark}[thm]{Remark}

\renewcommand{\P}{\mathbb{P}}
\newcommand{\un}{\mathbf{1}}
\newcommand{\R}{\mathbb{R}}

\newcommand{\Z}{\mathbb{Z}}

\newcommand{\E}{\mathbb{E}}

\newcommand{\ep}{\epsilon}

\newcommand{\leqps}{\stackrel{\mbox{\scriptsize a.s.}}{\leq}}

\title[Sharpness for monotone absorbing IPS]{Sharpness for monotone absorbing Interacting Particle Systems}
\author{Jean Bérard }\address{Institut de Recherche Mathématique Avancée, 
UMR 7501 Université de Strasbourg et CNRS, 
7 rue René-Descartes, 67000 Strasbourg, France}\email{jean.berard@math.unistra.fr} 
\author{Barbara Dembin}\address{Institut de Recherche Mathématique Avancée, 
UMR 7501 Université de Strasbourg et CNRS, 
7 rue René-Descartes, 67000 Strasbourg, France}\email{barbara.dembin@math.unistra.fr}
\author{Laure Marêché}\address{Institut de Recherche Mathématique Avancée, 
UMR 7501 Université de Strasbourg et CNRS, 
7 rue René-Descartes, 67000 Strasbourg, France}\email{laure.mareche@math.unistra.fr}

\begin{document}

\begin{abstract}
We prove a sharpness result for the dynamics of finite-range Interacting Particle Systems (IPS) on $\{0,1\}^{\Z^d}$, which generalizes to a whole class of IPS the sharpness result for the phase transition of the contact process obtained by Bezuidenhout and Grimmett~\cite{BezuidenhoutGrimmett1991}. More precisely, starting from an IPS that is monotone, ergodic, and which admits the all-zero configuration as an absorbing state, we prove that there exists an arbitrarily small perturbation of the dynamics which leads to an \emph{exponentially} ergodic IPS. This also extends the sharpness result previously established for (discrete-time) probabilistic cellular automata in \cite{Har} to the continuous-time setting of IPS.
\end{abstract}

\maketitle
\textbf{MSC classification:} Primary 60K35; Secondary 82C22.

\textbf{Keywords:} interacting particle systems, sharpness, exponential ergodicity, OSSS inequality.

\section{Introduction}

\subsection{Sharpness for IPS}

The notion of \emph{sharpness} plays a central role in percolation theory. Most percolation models exhibit a \emph{phase transition}, characterized by a \emph{subcritical phase}, in which no infinite cluster exists, and a \emph{supercritical phase}, where such an infinite cluster does appear (see \cite{Grimmett}). In the subcritical phase, the absence of an infinite cluster implies that the probability of a given point being connected to distance $n$ tends to zero as $n \to \infty$.
For percolation models with short-range interactions, the subcritical phase is said to be \emph{sharp} if this decay occurs exponentially fast in $n$, for any parameter value within the subcritical phase. This property was first established in the setting of Bernoulli percolation \cite{AizenmanBarsky,Menshikov}, and more recently given a more robust and general proof by Duminil-Copin, Raoufi, and Tassion using the \emph{OSSS inequality} \cite{OSSSdiscret}. Their approach has since been successfully applied to a wide class of percolation models (see for instance \cite{DCRT2,DCRT,GhoshRahul,Hartarsky2018}) and has found applications in other related areas (for instance probabilistic cellular automata, see \cite{Har}, which was a key inspiration for the present paper).

The question of sharpness has also been investigated in the setting of \emph{interacting particle systems (IPS)}, which is the focus of this work. An IPS is a Markov process in which every element of $\mathbb{Z}^d$ (they are called \emph{sites}) is given a state, and the states evolve in continuous time as follows. The timeline of each site is independently equipped with a Poisson point process, and when this process has a point, the site updates its state randomly according to a distribution depending only on the current configuration within a finite neighborhood. In this paper, we focus on monotone IPS on $\{0,1\}^{\Z^d}$; that is, the set of possible states at each site is $\{0,1\}$, and the Markov process dynamics preserves the usual partial order between probability distributions on $\{0,1\}^{\Z^d}$.

We say that an IPS is \emph{ergodic} if it forgets its initial condition over time; that is, for any initial configuration, the probability distribution of the configuration at time $t$ converges as $t$ goes to infinity to a limiting distribution that does not depend on the initial configuration. When in addition the IPS is monotone, ergodicity is equivalent to the fact that the probability distribution of the state of any given site, converges to a limiting distribution that does not depend of the initial configuration.
If this convergence occurs exponentially fast, we say the IPS is \emph{exponentially ergodic}.

By analogy with percolation, monotone exponentially ergodic IPS can be viewed as \emph{subcritical}, while monotone ergodic but not exponentially ergodic IPS may be thought of as \emph{critical}\footnote{One can prove that if a monotone IPS is ergodic but not exponentially ergodic, then the speed of convergence is at most polynomial
 (this can be done along the lines of the proof of Theorem 3.12 of \cite{Martinelli}). This is analogous to the critical phase of percolation, where the probability of connection to distance $n$ is expected to decay polynomially.}. Monotone non-ergodic systems, on the other hand, can be seen as \emph{supercritical}, as they allow information to ``percolate'' to infinity.
By sharpness in the context of monotone IPS, we mean that for every ergodic IPS we can find an exponentially ergodic IPS arbitrarily close to it\footnote{This is analoguous to the notion of sharpness for  Bernoulli percolation with parameter $p$. Indeed if we prove that, for every $p$ such that there is no percolation, there is exponential decay at $p-\ep$ for every $\ep > 0$, then we can prove that there is exponential decay in the full subcritical regime, using monotonicity.}.

 An important IPS for which sharpness has received much attention is the \emph{contact process}, a model that describes the spread of an infection. In this IPS, a site in state $1$ (infected) recovers into state $0$ (healthy) at rate $1$, while each site in state $1$ ``contaminates'' each of its neighbors at rate $\lambda$, for some parameter $\lambda>0$ (see Part I of \cite{Lig2} for an introduction). This model admits the all-0 configuration as an absorbing state, and is ergodic when, starting from any initial configuration, the probability that a given site is in state $1$ tends to $0$ as time goes to infinity. It is well-known that there exists a critical threshold $\lambda_c$ so that when $\lambda<\lambda_c$ the process is ergodic while if $\lambda>\lambda_c$ it is not ergodic. An important result of Bezuidenhout and Grimmett \cite{BezuidenhoutGrimmett1990} shows the contact process is ergodic at $\lambda_c$; this was later generalized to more general monotone IPS by Bezuidenhout and Gray \cite{BezuidenhoutGray1994}. Sharpness for the contact process means that for all $\lambda < \lambda_c$, starting from the all-$1$ configuration, the probability that a given site is in state $1$ tends to $0$ exponentially quickly, which was first proven by Bezuidenhout and Grimmett \cite{BezuidenhoutGrimmett1991}. Much later, Swart \cite{Swart} and Beekenkamp \cite{beekenkamp} gave other proofs of this\footnote{Technically, the sharpness results in \cite{BezuidenhoutGrimmett1991,Swart,beekenkamp} are stated in terms of the probability that there are still some $1$s at time $t$ starting from a single $1$ at time $0$, but this probability is equal to the probability a given site is in state $1$ starting from an all-$1$ initial configuration by a duality property (see (1.7) in \cite{Lig2}).}.

In this paper, we give a proof of sharpness for a general class of IPS which includes the contact process as a special case (in a similar spirit as the generalization of the results of \cite{BezuidenhoutGrimmett1990} by \cite{BezuidenhoutGray1994}).
More precisely, we study monotone \emph{absorbing} IPS, which means that the limiting distribution is a Dirac mass concentrated on the all-$0$ configuration, and prove that any ergodic IPS of this class can be perturbed to become exponentially ergodic. 
We also prove some topological properties of the set of ergodic IPS within this class.

\subsection{Model and results}

 For any $\xi \in \{0,1\}^{\mathbb{Z}^d}$, $x\in\mathbb{Z}^d$, we denote by $\xi(x)$ the value of $\xi$ at $x$. For any $m\in\mathbb{N}$, we denote $\Lambda_m\coloneqq\{-m,...,m\}^d$. We consider finite-range interactions, so we let $R\in\mathbb{N}$ and define the \emph{transition rates} as functions $c_0$ and $c_1$ $: \{0,1\}^{\Lambda_R} \to \mathbb{R}_+$. For $\xi \in \{0,1\}^{\mathbb{Z}^d}$, $i=0$ or $1$, we will write $c_i(\xi)$ for $c_i(\xi_{|\Lambda_R})$, where $\xi_{|\Lambda_R}$ is the restriction of $\xi$ to $\Lambda_R$. 
The infinitesimal generator of the IPS\footnote{See \cite{Lig} p. 122 for the definition of an IPS through its generator. Note that their definition is via local flip rates $c(\xi)$, where the generator is written as
$Lf(\xi) = \sum_{x \in \Z^d} c(\tau_x \xi) (f(\xi^{x,1-\xi(x)})-f(\xi)).$
This is equivalent to our setting if we choose $c(\xi) = c_1(\xi)$ when $\xi(0)=0$, and $c(\xi)=c_0(\xi)$ when $\xi(0)=1$.} is given by
\begin{equation}\label{e:def-L}Lf(\xi) \coloneqq \sum_{x \in \Z^d} (c_1(\tau_x \xi) (f(\xi^{x,1})-f(\xi)) +c_0(\tau_x \xi) (f(\xi^{x,0})-f(\xi)) ),\end{equation}
where $f \ : \ \{0,1\}^{\Z^d} \to \R$ depends on finitely many coordinates,  and, for $x \in \Z^d$ and $\xi \in \{0,1\}^{\Z^d}$,  
 $\tau_x \xi \in \{0,1\}^{\Z^d}$ is given by $(\tau_x \xi)(y) \coloneqq \xi(y-x)$, and where $\xi^{x,1} \in \{0,1\}^{\Z^d}$ satisfies $\xi^{x,1}(y) = \xi(y)$ for $y \neq x$, $\xi^{x,1}(x)=1$ ($\xi^{x,0}$ is defined similarly). 

More generally, we consider a perturbation of the previous dynamics, whose generator is given, for $\varepsilon > 0$, by 
\begin{equation}\label{e:def-Leps}L^{\varepsilon} f(\xi) \coloneqq L f (\xi) + \varepsilon \sum_{x\in\mathbb{Z}^d}(f(\xi^{x,0})-f(\xi)),\end{equation}
which amounts to replacing $c_0$ by $c_0 + \varepsilon$ in \eqref{e:def-L}.

Denote by $\mathbf{0}$ the configuration with $0$ at every site (and $\mathbf{1}$ is defined similarly). 
 For any $\xi,\eta \in \{0,1\}^{\mathbb{Z}^d}$, we write $\xi \leq \eta$ when for all $x\in\mathbb{Z}^d$ we have $\xi(x)\leq\eta(x)$. 

We assume that rates enjoy the following properties: 
\begin{itemize}
\item[(i)]\label{prop:mon} monotonicity:  when $\xi \leq \eta$, $c_1(\xi) \leq c_1(\eta)$ and $c_0(\xi) \geq c_0(\eta)$;
\item[(ii)] $\mathbf{0}$ is absorbing: $c_1(\mathbf{0}) = 0$.\label{prop:abs}
\end{itemize}  
 We denote by $\mathscr{L}$ the set of generators $L$ defined as above whose rates satisfy (i) and (ii). 

 In the specific context of an IPS $(X_t)_{t \geq 0}$ with generator in $\mathscr{L}$, 
ergodicity is equivalent to the fact that, starting at $X_0=\mathbf 1$, we have that $\lim_{t \to +\infty} \P(X_t(0)=1) = 0$.  We say that the IPS has \emph{exponential ergodicity} if, moreover, there exist $C,c>0$ such that for all $t\ge0$
\begin{equation}\label{eq:exponential ergodicity}
    \P(X_t(0)=1)\le Ce^{-ct}.
\end{equation}

\begin{thm}\label{t:BJL-1}
Assume $L$ is a generator in $\mathscr{L}$, hence monotone with $\mathbf{0}$ absorbing, such that the IPS with generator $L$ defined in \eqref{e:def-L} is ergodic. Then, for all $\varepsilon > 0$,  the perturbed IPS with generator $L^{\varepsilon}$ defined in \eqref{e:def-Leps} is exponentially ergodic.
\end{thm}

Note that the constants $C,c$ in \eqref{eq:exponential ergodicity} associated with the statement of exponential ergodicity in Theorem \ref{t:BJL-1} may depend on $\varepsilon$.

 The generator $L$ is characterized by the couple $(c_0,c_1)$, which can be seen as an element of $\mathbb{R}^n$ for some integer\footnote{The $c_i$ are functions $: \{0,1\}^{\Lambda_R} \to \mathbb{R}$, so they can be seen as tuples whose length is the cardinal of $\{0,1\}^{\Lambda_R}$.} $n$. This yields a natural topology on such generators. In the following, we will consider the topology \emph{induced on $\mathscr{L}$}. For any set $S \subset \mathscr{L}$, let $\mathrm{Int}(S)$ and $\mathrm{Cl}( S)$ denote respectively the interior and the closure of the set $S$ with respect to this induced topology.
\begin{cor}\label{cor_topology}Let $S$ be the set of generators in $\mathscr{L}$ such that the associated IPS is ergodic. Let $L$ be in the interior of $S$, then the IPS associated to $L$ is exponentially ergodic. Moreover, we have $S\subset \mathrm{Cl(\mathrm{Int}(S)})$.
    
\end{cor}
We do not give a full proof of Corollary \ref{cor_topology}. One can argue as in the proof of Theorem 1.2 of \cite{Har}, but replacing their important Proposition 2.5 by Theorem 2.2 of Crawford--De Roeck \cite{Crawford2018}.

\subsection{Sketch of proof}

The IPS admits a \emph{graphical construction}, an encoding of the system's dynamics using a space-time Poisson point process. Specifically, the graphical construction is governed by a Poisson point process on $\mathbb{Z}^d \times \mathbb{R}_+ \times [0,M]$, where $M$ is a suitable number. Here $\mathbb{Z}^d$ and $\mathbb{R}_+$ respectively encode space and time, and the interval $[0,M]$ is used to determine the update at each space-time point.

Each atom $(x, t, u)$ in this Poisson point process corresponds to a potential update at site $x$ and time $t$, while $u\in[0,M]$ is fed into a function to determine the outcome of the update, in such a way the outcome has the correct distribution. 
The state of the system at any time $t$ and site $x$, denoted $X_t(x)$, is then determined by iteratively applying these updates in chronological order, starting from the initial configuration.

We now explain how we perturb the dynamics. Each point in the Poisson point process is initially marked with the label $A$. Then, independently for each point, we change its mark from $A$ to $B$ with probability $h \in [0,1]$, which serves as our perturbation parameter.
The perturbed dynamics is defined as follows:
at points marked $A$, we apply the update associated to the unperturbed IPS. At points marked $B$, we override the update and set the state of the site to $0$.

This defines a new perturbed generator $L_h$, which corresponds to a time-rescaled version of the original generator $L_\varepsilon$ (see the formal definition in Equation~\eqref{e:def-Lh}). Importantly, exponential ergodicity for $L_h$ implies exponential ergodicity for $L_\varepsilon$. From now on, we work with $L_h$ which is more amenable to analysis.

Our primary analytical tool is the \emph{continuous version of the OSSS inequality}\footnote{Although we use the continuous form of the OSSS inequality, a similar argument could be carried out using the discrete version as in the proof of \cite[Theorem 5.9]{beekenkamp}.}, established in \cite{LasPecYog}. We apply the OSSS inequality to the function $f\coloneqq \un_{\{ X_t(0)=1\}}$ where $(X_t)_{t\ge 0}$ is the IPS with generator $L_h$ starting at $X_0=\mathbf 1$. The OSSS inequality is associated with an exploration algorithm of the graphical construction to determine the value of $f$. A central notion of the OSSS inequality is the \emph{influence}; 
in the continuous setting, it corresponds to the probability that adding a point to the graphical construction changes the value of $f$. 
The main difficulty lies in the fact that a point in the graphical construction can be \emph{pivotal}—meaning its presence changes the value of $f$—without its \emph{mark} (either $A$ or $B$) being pivotal\footnote{For instance, consider adding a point at a pivotal space-time location—i.e., a site where changing its value affects $f$—where the current value is $1$. Suppose the update rule of the IPS with mark $A$ results in $0$. In this case, the output is $0$ regardless of whether the mark is $A$ or $B$, so the mark is not pivotal, even though the point’s presence is.}. However, when relating pivotality to the derivative with respect to the perturbation parameter $h$, it is the pivotality of the mark that matters.
As a result, the influence tied to the perturbation cannot be fully understood through a pathwise comparison alone. Instead, the analysis must be conducted at the level of expectations and averages over the randomness in the graphical construction. This difficulty was not present in the discrete setting in \cite{Har}.

\subsection{Related works}

As mentioned earlier, our result generalizes the sharpness property of the contact process, which was originally proven by Bezuidenhout and Grimmett in \cite{BezuidenhoutGrimmett1991}. The proof in \cite{BezuidenhoutGrimmett1991} relies on pivotality estimates, but does not take advantage of the (as yet unknown) OSSS machinery. In \cite{Swart}, Swart gave a much simpler proof of sharpness for the contact process, based on very different ideas, namely the study of a suitable harmonic function. The proof of Beekenkamp \cite[Theorem 5.9]{beekenkamp} uses the OSSS inequality, so we expand on it a little more. The contact process admits a graphical construction governed by a marked Poisson point process, where marks represent either recoveries or arrows indicating the transmission of infection to a neighboring site.  In this framework, \cite{beekenkamp} relies on the discrete OSSS inequality thanks to a suitable discrete partitioning of the time variable. A key distinction between his setting and ours lies in the notion of pivotality: in the contact process, the pivotality of the presence of a point is equivalent to the pivotality of its mark (e.g., changing a recovery into an infection arrow, or vice versa). This equivalence allows for a straightforward connection between influence and the derivative with respect to $\lambda$. In contrast, this property does not hold in general for our model, where a point's presence may be pivotal without its mark being influential. Thus, we must rely on more delicate modification arguments to overcome this difficulty and we recover the results of \cite{beekenkamp} as a special case of our main theorem. 

 Finally, in \cite{Har}, Hartarsky investigates the sharpness of phase transitions in monotone probabilistic cellular automata (PCA) with $0$ as an absorbing state, which corresponds to the discrete-time analogue of our model. He considers a parametrized family of automata constructed from a given probabilistic automaton $\mathcal A$, where at each site and time step, the update follows the rules of $\mathcal A$ with probability $1-p$, and outputs $0$ with probability $p$. Hartarsky proves that this family exhibits a sharp phase transition: in the (possibly empty) supercritical phase, the automaton does not admit $ \delta_0$ as the unique invariant measure, while in the subcritical phase, the system has exponential ergodicity. 
Our work extends this result to the continuous setting of IPS. However, this extension is non-trivial and requires techniques that are not needed in the discrete-time setting of PCA. Indeed, in the case of PCA, the influence can be easily controlled by the pivotality, while this control is more difficult to achieve in our continuous setting as explained in the sketch of proof given above.

\subsection{Open questions}
In this work, we have focused exclusively on the case of monotone absorbing IPS. For this class, it is straightforward to identify perturbations that accelerate convergence (here, by adding updates leading to $0$). However, for general monotone IPS, it remains unclear which perturbations, if any, may speed up convergence. A natural question arises: can this result be extended beyond systems with absorbing states? More specifically, does a monotone ergodic automaton always admit an arbitrarily small perturbation that is exponentially ergodic?

\subsection{Organization of the paper}
In Section \ref{Sec: coupling}, we introduce the graphical construction and the definition of pivotality together with basic properties and estimates. Section \ref{sec:OSSS} is devoted to the exposition of the continuous OSSS framework from \cite{LasPecYog}, the definition of the exploration process and the proof of Theorem \ref{t:BJL-1}.
 
\section{Coupling structure and pivotality}\label{Sec: coupling}

\subsection{Graphical construction}

\subsubsection{Construction of the dynamics}\label{sec_cons}

We use a version of the so-called graphical construction of the dynamics based on a Poisson process, in the spirit of \cite{Harr}. See e.g. \cite{Swa} (Section 4.3) for more details on this kind of construction in a general setup. 

In the sequel, we denote by $L$ an infinitesimal generator satisfying the assumptions of Theorem \ref{t:BJL-1}, i.e. $L\in \mathscr{L}$ and the resulting IPS is ergodic. We define two constants $C_0$ and $C_1$ by $C_0\coloneqq\max\{c_0(\xi)\,|\,\xi\in\{0,1\}^{\mathbb{Z}^d}\}$ and $C_1\coloneqq\max\{c_1(\xi)\,|\,\xi\in\{0,1\}^{\mathbb{Z}^d}\}$. Note that, thanks to the monotonicity condition (i), we have that $C_0 = c_0(\mathbf{0})$ and $C_1 = c_1(\mathbf{1})$. Let then $M\coloneqq C_0+C_1$.
 \begin{remark}Note that the value of $c_0(\xi)$ when $\xi(0)=0$ (resp. the value of $c_1(\xi)$ when $\xi(0)=1$) has no impact on the dynamics. Hence, we can assume without loss of generality that
\begin{itemize}
\item[(iii)] $c_1(\mathbf{1})>0$
\end{itemize}
and conditions (i) and (ii) still hold. From now on, we will assume that (i), (ii) and (iii) hold.

\end{remark}

In the proofs, it will be more convenient to work with the following perturbed generator rather than $L^\ep$: for $h\in[0,1]$,
\begin{equation}\label{e:def-Lh}L_h f(\xi) \coloneqq (1-h) L f (\xi) + M h \sum_{x\in\mathbb{Z}^d}(f(\xi^{x,0})-f(\xi)).\end{equation}

Let $\mathscr{S} \coloneqq \Z^d \times \R_+ \times [0,M] \times \{A , B \}$, and for $h \in [0,1]$, define a locally finite positive measure $\lambda_h$ on $\mathscr{S}$ by
$\lambda_h \coloneqq M \left(\mbox{counting} \otimes \mbox{Lebesgue} \otimes \mbox{uniform} \otimes ((1-h) \delta_{\{A\}} + h \delta_{\{B\}}) \right).$
We then consider a Poisson process $\mathscr{P}$ on $\mathscr{S}$ whose intensity measure is $\lambda_h$; $\mathscr{P}$ is viewed as a random variable on a probability space $(\Omega, \mathcal{F} ,\P_h)$.

Denote by $\pi$ the canonical projection from $\mathscr{S}$ to $\Z^d \times \R_+$, i.e. $\pi(x,t,u,w) \coloneqq (x,t)$. 

Starting from an initial condition $X_0 \coloneqq \xi \in \{0,1\}^{\Z^d}$, the value of $X_t(x)$ for any $t > 0$ and $x \in \Z^d$ is defined as follows. For a given $x$, those $t$ at which the value of $X_t(x)$ may change are precisely those satisfying $(x,t) \in \pi(\mathscr{P})$. Given such a $t$ and the corresponding element $(x,t,u,w) \in \mathscr{P}$, the value of $X_t(x)$ is defined from the value of $X_{t-}$ according to the following rules: 
\begin{itemize}
\item if $w = B$, then $X_t(x)\coloneqq 0$; 
\item if  $w = A$, then:
 \begin{itemize}
 \item if $u \in [0,c_0(\tau_x X_{t-})[$, then $X_t(x) \coloneqq0$; 
 \item if $u \in [c_0(\tau_x X_{t-}), C_0 + C_1 - c_1(\tau_x X_{t-})[$, then $X_t(x) \coloneqq X_{t-}(x)$; 
 \item  if $u \in [C_0 + C_1 - c_1(\tau_x X_{t-}),C_0 + C_1]$, then $X_t(x) \coloneqq 1$.
 \end{itemize}
 \end{itemize}

\begin{prop} Let $h\in[0,1].$ The graphical construction is well-defined with probability one simultaneously for all $x\in\mathbb{Z}^d,t \geq 0,\xi\in\{0,1\}^{\mathbb{Z}^d}$.
    Moreover,  $X_t = (X_t(x))_{x \in \Z^d}$ is a continuous-time Markov process on $\{0,1\}^{\Z^d}$, starting at $X_0 = \xi$, and whose infinitesimal generator is given by \eqref{e:def-Lh}.
\end{prop}

To prove this proposition, one needs to ensure that with probability one, simultaneously for every $x,t,\xi,h$, the number of elements $(y,s,u,w)\in\mathscr{P}$ that are relevant to determine the value of $X_t(x)$ is finite (this will be a consequence of Lemma \ref{lem : taille It}). Moreover, with probability one, any pair of elements in $\mathscr{P}$ leads to two distinct values for the time variable. This allows one to unambiguously define $X_t(x)$ for all $t>0$, $x\in\mathbb{Z}^d$, $\xi\in\{0,1\}^{\mathbb{Z}^d}$. One can prove $(X_t)_{t \geq 0}$ is a Markov process with generator $L$ as in Proposition 2.7 of \cite{Swa}.

Now, given $x \in \Z^d$, $t \geq 0$ and $s \geq t$, we define ${}^{x,t,1}X_s$ to be the value of $X_s$ obtained by altering the previous construction, forcing the value of $X_t(x)$ to be $1$. For $s<t$, we just let ${}^{x,t,1}X_s\coloneqq X_s$. We define ${}^{x,t,0}X_s$ in a similar way.

 The following monotonicity properties are straightforward consequence of the graphical construction.
\begin{lem}\label{lem: monotonicity in xi}Let $t\ge 0$ and $x\in\Z^d$. The resulting $X_t(x)$ is monotonically increasing as a (random) function of the initial configuration $\xi$.  Moreover, we have almost surely for any $s\ge 0$ that ${}^{x,t,0}X_s \leq X_s \leq {}^{x,t,1}X_s$.
\end{lem}

\subsubsection{Bound on the set of possibly influential sites}
\label{sec_influence}

Given $T > 0$, we now define a family of subsets of $\Z^d$, denoted 
$(I_t^T)_{t \in [0,T]}$ so that ``$I_t^ T$ contains every site whose state at time $t$ may influence the state of the origin at time $T$''. The goal of this section is to obtain upper bounds on the size of $I_t^ T$.

 The family $(I_t^T)_{t \in [0,T]}$ is defined by going backward in time in the following way. Start with $I_{T}^T \coloneqq \{ 0 \}$, and let $t_0 \coloneqq T$. For $j \geq 0$, assuming that $t_j$ is defined and that $I_t^T$ is defined for $t_j \leq t \leq T$, define $t_{j+1} \coloneqq \sup \{ t \in ]0, t_j[ \, | \,  (x,t) \in \pi(\mathscr{P}) \mbox{ for some } x \in I^T_{t_j}  \}$, with the convention that $\sup \emptyset = 0$. Set $I_t^T \coloneqq I_{t_j}^T$ for $t \in ]t_{j+1},t_j[$. If $t_{j+1}>0$, set $I^T_{t_{j+1}} \coloneqq I^T_{t_j} \cup (x_{j+1} + \Lambda_R)$, where $x_{j+1}$ is the (a.s. unique) site realizing the supremum in the definition of $t_{j+1}$, and, in the case where $t_{j+1}=0$, let $I^T_{t_{j+1}}\coloneqq I^T_{t_j}$.

The reason we are interested in these sets is the following lemma  which says that the sites which may influence the outcome of $X_T(0)$ are included in the set of influential sites $I^T_0.$   We postpone its proof to the end of this section.
\begin{lem}\label{lem: pivot}Let $x\in\Z^d$.  Almost surely, if $x \notin I^T_0$, then, for all $t \in [0,T],$ ${}^{x,t,0}X_T(0) = {}^{x,t,1}X_T(0)$.
    
\end{lem}

By comparing the dynamics of $I_t^T$ with a branching random walk, one can obtain the following localization result of $I_0^T$ around the origin.
\begin{lem}\label{lem : taille It}There exist constants $\kappa_1\geq 1$, $\kappa_2>0$ depending only on $T$, $M$ and $R$ such that,  for any $x\in\mathbb Z^d$, 
\begin{equation}\label{e:controle-spread}
    \P_h(x\in I_0^T)\le \kappa_1\exp(-\kappa_2 \|x\|_\infty).
\end{equation}
    \end{lem}

\begin{proof}

One can compare the dynamics of $I_t^T$ as $t$ decreases from $T$ to $0$ with the spread of a branching random walk in continuous time, which branches at rate $M$, and such that, when a particle located at site $y$ branches, it dies and new particles located at sites $y+ \Lambda_R$ are added to the current population. (This comparison amounts to ignoring the fact that $I^T_{t_j}$ may already contain sites in $x_{j+1}+ \Lambda_R$, thus providing an upper bound on the original process.)
 We first observe that there almost surely exists a $j$ such that $t_{j+1}=0$, as a consequence of classical results, e.g. Equation (4) page 108 Chapter III in \cite{AthNey}, which shows that the total number of particles in the branching process after time $T$ has a finite expected value equal to $e^{M (| \Lambda_R|-1) T}$, where for any set $\mathcal{E}$, we denote by $|\mathcal{E}|$ the cardinal of $\mathcal{E}$.

 Let us now derive the localization result on $I_0^T$ around the origin. Using a suitable many-to-one formula (see e.g. Proposition 3.3 in \cite{BanDelMarTra}), we have that the expected number of particles after time $T$ at a site $x$ is equal to $e^{M (|\Lambda_R| -1)T} \cdot \P(\beta_T = x)$, where $(\beta_t)_{t \geq 0}$ is a continuous-time random walk on $\Z^d$ starting at $0$ with constant jump-rate equal to $M | \Lambda_R|$, and whose jump distribution is the uniform distribution on $\Lambda_R$. 
Since the total number of jumps performed by the random walk follows a Poisson distribution with parameter $M |\Lambda_R| T$, and since the jump distribution is deterministically bounded, \eqref{e:controle-spread} is easily deduced from classical tail estimates on the Poisson distribution.

\end{proof}

 \begin{proof}[Proof of Lemma  \ref{lem: pivot}] The proof of Lemma \ref{lem : taille It} ensures there a.s. exists $j$ so that $t_{j+1}=0$. One can prove Lemma \ref{lem: pivot} by showing by induction on decreasing $i$ (which means that time goes forward) that for any $x \notin I^T_0$, $t \in [0,T[$, then for all $s\in[t_{i+1},t_{i}]$ the configurations ${}^{x,t,0}X_s$ and ${}^{x,t,1}X_s$ coincide on $I_s^T$. 
\end{proof}

\subsection{Properties of $\theta_T$ and pivotality}\label{sec:derivative_theta}

Given a time-horizon $T \geq 0$, we define \[\theta_T(h) \coloneqq \P_h(X_T(0) = 1),\] where the initial condition is $X_0 = \mathbf{1}$. In this section, we will prove regularity properties of $\theta_T(h)$ and give an expression for its derivative. Since there are no readily available formulas for the derivative when working with infinite measures, we will first compute the derivative for a truncated version of the model and then prove that the latter converges towards $\theta'_T(h)$.

\begin{prop}\label{prop_theta(T)}
For fixed $h$, $T \mapsto \theta_T(h)$ is a continuous and monotonically non-increasing function. 
\end{prop}

\begin{proof}
We deal with continuity first. Observe that, by definition of the graphical construction, for $T_1 < T_2$, $\P_h(X_{T_2}(0) \neq X_{T_1}(0)) \leq \P_h(\pi(\mathscr{P}) \cap (\{ 0\} \times ]T_1,T_2])\neq \emptyset )) = 1- e^{-M(T_2-T_1)}$,
so $|\theta_{T_1}(h) - \theta_{T_2}(h) | \leq 1- e^{-M(T_2-T_1)}$.

We now deal with the monotonicity property. Consider $0 \leq T_1 \leq T_2$. By  Lemma \ref{lem: monotonicity in xi}, we have that $X_t(0)$ is monotonically increasing as a function of the initial condition, so that $\P_h(X_t(0)=1)$ is a monotonically increasing function of the initial condition as well. Observing that, starting from $X_0 = \mathbf{1}$, we have $X_{T_2-T_1} \leq \mathbf{1}$, and applying the Markov property, we deduce that $\P_h(X_{T_2}(0) = 1 | X_{T_2-T_1}) \leqps  \theta_{T_1}(h)$, so that, taking the expectation, we deduce that $\theta_{T_2}(h) \leq \theta_{T_1}(h)$.    
\end{proof}

Let us now introduce the key notion of pivotality that will appear in the expression of the $h$-derivative of $\theta_T$. We say that a triple $(x,t,u)$, where $x \in \Z^d$, $t \in ]0,T]$, and $u \in [0,M]$, is $T-$pivotal, when, starting with $X_0 = \mathbf{1}$, the following conditions are met:
\begin{itemize}
\item[$(a)$] $X_t(x)[\mathscr{P} \cup \{ (x,t,u,A) \}] = 1$;
\item[$(b)$] ${}^{x,t,0}X_T(0) = 0$ and ${}^{x,t,1}X_T(0) = 1$;
\end{itemize}
where the notation $X_t(x)[\mathscr{P} \cup \{ (x,t,u,A) \}]$ denotes the value of $X_t(x)$ produced by the graphical construction when $\mathscr{P}$ is replaced by $\mathscr{P} \cup \{ (x,t,u,A) \}$. Condition $(a)$ above is necessary to ensure that changing the mark from $A$ to $B$ at $(x,t,u)$ changes the value of $X_t(x)$.

 Thanks to Lemma \ref{lem: pivot}, if $(x,t,u)$ is $T-$pivotal, then almost surely $x \in I^T_0$, so,
\begin{equation}\label{e:queue-pivotÉ} \P_h ((x,t,u) \mbox{ is $T-$pivotal}) \leq \P_h(x\in I_0^T).\end{equation}
Moreover, 
using \eqref{e:controle-spread}, we have that 
\begin{equation}\label{e:queue-pivot} \P_h ((x,t,u) \mbox{ is $T-$pivotal}) \leq \kappa_1 \exp(-\kappa_2 ||x||_{\infty}).\end{equation}

\begin{lem}\label{l:ppivot-continue-h}
For fixed $T,x,t,u$, the probability $\P_h ((x,t,u) \mbox{ is $T-$pivotal})$ is a continuous function of $h$.
\end{lem}

\begin{proof}In this proof, we need to couple the graphical constructions corresponding to various values of $h$. To do so, we work with uniform random marks in $[0,1]$ instead of $\{A,B\}$ marks and deduce from them $\{A,B\}$ marks by comparing the $[0,1]$-valued marks with $h$.
More precisely, we consider a Poisson process $\mathscr{Q}$ on $\Z^d \times \R_+ \times [0,M]$ with intensity measure defined as $M($counting$\otimes$Lebesgue$\otimes$uniform$)$ and i.i.d. marks following the uniform distribution on $[0,1]$. For every $h \in [0,1]$, define $\mathscr{P}_h$ as the point process obtained from $\mathscr{Q}$ as follows: if $(y,s,v) \in \mathscr{Q}$ and the corresponding mark is $\geq h$, then $(y,s,v,A)  \in \mathscr{P}_h$, while if the corresponding mark is $<h$, then $(y,s,v,B) \in \mathscr{P}_h$. This way, $\mathscr{P}_h$ has exactly the law of $\mathscr{P}$ with respect to $\P_h$. 

Given $(x,t,u)$, define $V(x,t,u,h)$ as the event that $(x,t,u)$ is $T-$pivotal for the graphical construction based on the point process $\mathscr{P}_{h}$. For any $h_0 \in [0,1]$, we then have that, almost surely, $\un_{V(x,t,u,h)}$ converges to $\un_{V(x,t,u,h_0)}$ as $h \to h_0$. To see this, observe that the definition of the sets $I^T_t$ does not involve the marks, so the same set is obtained for every value of $h$. As a consequence, we only have to look at the effect of the limit $h \to h_0$ along a finite set of stochastic updates, for which the convergence is obvious provided that none of the $[0,1]$-valued marks in $\mathscr{Q}$ is exactly equal to $h_0$, this last condition being satisfied almost surely.  

The result then follows from Lebesgue's dominated convergence theorem.
\end{proof}

Using a similar coupling as in the proof above and noting it makes $X_T(0)$ non-increasing in $h$, one can prove the following monotonicity property.

\begin{lem}\label{lem: monotonicity in h}For fixed $T$, the function $h\mapsto \theta_T(h)$ is non-increasing.
    
\end{lem}

To prove the Russo formula, we introduce a truncated version of the dynamics, defined with the points in $\mathscr{S}_m^T \coloneqq \Lambda_m\times [0,T] \times [0,M] \times \{A , B \}$. We use the graphical construction with the truncated Poisson process  $\mathscr{P}_m^T \coloneqq\mathscr{P} \cap \mathscr{S}_m^T$, whose intensity measure is
$\lambda_{h,m}^T \coloneqq M \left(\mbox{counting} \otimes \mbox{Lebesgue} \otimes \mbox{uniform} \otimes ((1-h) \delta_{\{A\}} + h \delta_{\{B\}})\right).$

We define $X_t^m(x)$, for all $0 \leq t \leq T$ and $x \in \Lambda_m$ using a graphical construction based on $\mathscr{P}_{m}^{T}$, exactly as $X_t(x)$ is defined from $\mathscr{P}$, with the exception that we have to specify boundary conditions: as soon as the state at a site $x \notin\Lambda_m$ is needed in the construction, we declare it to be $0$.

We then define $\theta_T^m(h)$, and the notion of $T\stackrel{m}{-}$pivotality, analogously to the original (not truncated) case. Observe that, as in the original case (see Lemma \ref{lem: pivot}), the set $I^T_0$ also bounds the set of possible influential sites for the truncated dynamics.

\begin{lem}\label{limite-m}
We have that, for fixed $T,x,t,u$,
\begin{equation}\label{e:cv-theta-m}\lim_{m \to +\infty} \theta_T^m(h)  = \theta_T(h).\end{equation}
and that
\begin{equation}\label{e:cv-ppivot-m}\lim_{m \to +\infty} \P_h \left((x,t,u) \mbox{ is $T\stackrel{m}{-}$pivotal}\right) = \P_h \left((x,t,u) \mbox{ is $T-$pivotal} \right).\end{equation}
\end{lem}

\begin{proof}
Observe that,  for all large enough $m$, $I_0^T \subset \Lambda_m$. As a consequence, almost surely, 
$\un_{\{X_T^m(0) = 1\}}$ converges to $\un_{\{X_T(0) = 1\}}$ as $m$ goes to infinity, and \eqref{e:cv-theta-m} is a consequence of Lebesgue's dominated convergence theorem. 
The proof of \eqref{e:cv-ppivot-m} is completely similar. 
\end{proof}

We start with Russo's formula in the truncated case.
 To state it, we use the notation $\widetilde{\mathscr{S}}_{m}^{T} \coloneqq \Lambda_m\times (0,T] \times [0,M]$ and the corresponding measure $\widetilde{\lambda}_{m}^{T}$ on  $\widetilde{\mathscr{S}}_{m}^{T}$ by
$\widetilde{\lambda}_{m}^{T} \coloneqq M  \left(\mbox{counting} \otimes \mbox{Lebesgue} \otimes \mbox{uniform}\right)$. We will use the   $\widetilde{\phantom{.}}$  notation whenever we do not consider marks.

\begin{lem}\label{l:Russo-continu-tronc} The derivative of $h \mapsto \theta_T^m(h)$ is given by
\begin{equation}\label{e:Russo-continu-tronc} (\theta_T^m)'(h) = - \int_{ \widetilde{{\mathscr{S}}}_{m}^{T}} \P_h \left((x,t,u) \mbox{ is $T\stackrel{m}{-}$pivotal} \right) d \widetilde{\lambda}_{m}^{T}(x,t,u).\end{equation}
\end{lem}

\begin{proof}
Write $\lambda_{h,m}^T$ as $\lambda_{h,m}^T = \mu + h \nu$, where 
$\mu \coloneqq M \left(\mbox{counting} \otimes \mbox{Lebesgue} \otimes \mbox{uniform} \otimes \delta_{\{A\}}\right)$ and 
$\nu \coloneqq M \left(\mbox{counting} \otimes \mbox{Lebesgue} \otimes \mbox{uniform} \otimes (\delta_{\{B\}} - \delta_{\{A\}}) \right)$.
Let $f \coloneqq \un_{ \left\{X_T^m(0) = 1 \right\}}$ when starting from $X_0^m = \mathbf{1}_{|\Lambda_m}$ ($f$ is viewed as a function of $\mathscr{P}_m^T$)
and using Theorem 19.1 in \cite{LasPen}, we have that 
$$ (\theta_T^m)'(h) = \int_{\mathscr{S}_m^T} \E_h \left(f(\mathscr{P}_m^T \cup \{(x,t,u,w) \}) - f(\mathscr{P}_m^T) \right) d \nu(x,t,u,w),$$
so that, in view of the definition of $\nu$, the above integral can be rewritten, after a little algebra, as  
$$\int_{\widetilde{{\mathscr{S}}}_m^T}  \E_h \left(f(\mathscr{P}_m^T \cup \{(x,t,u,B) \}) -  f(\mathscr{P}_m^T \cup \{(x,t,u,A) \}) \right) d \widetilde{\lambda}_m^T(x,t,u).$$
One then checks that 
$$f(\mathscr{P}_m^T \cup \{(x,t,u,A) \}) - f(\mathscr{P}_m^T \cup \{(x,t,u,B) \}) = \un_{\left\{ (x,t,u) \mbox{ is $T\stackrel{m}{-}$pivotal} \right\} }.$$
    \end{proof}

We now have Russo's formula for the original dynamics, using the notations  $\widetilde{\mathscr{S}}^{T} \coloneqq\Z^d\times (0,T] \times [0,M]$ and the corresponding measure $\widetilde{\lambda}^{T}$ on  $\widetilde{\mathscr{S}}^{T}$ defined by
$$\widetilde{\lambda}^{T} \coloneqq M \left(\mbox{counting} \otimes \mbox{Lebesgue} \otimes \mbox{uniform} \right).$$

\begin{prop}\label{p:Russo-continu-tronc} For all $T \geq 0$, the derivative of $h \mapsto \theta_T(h)$ is given by
\begin{equation}\label{e:Russo-continu} (\theta_T)'(h) = - \int_{ \widetilde{\mathscr{S}}^{T}} \P_h \left((x,t,u) \mbox{ is $T-$pivotal} \right) d \widetilde{\lambda}^T(x,t,u).\end{equation}
Moreover, we have that 
\begin{equation}\label{e:cv-derivee} \lim_{m \to +\infty} (\theta_T^m)'(h) = \theta_T'(h).\end{equation}
\end{prop}

\begin{proof}
Note that arguing as in the proof of Lemma \ref{l:ppivot-continue-h}, for all $m$, $\P_h \big((x,t,u) \mbox{ is $T\stackrel{m}{-}$pivotal}\big)$ is a (bounded) continuous function of $h$.
Since we also have that $\widetilde{\lambda}_{m}^{T}(x,t,u)$ is a positive measure with finite mass, we deduce from Lemma \ref{l:Russo-continu-tronc} that $h \mapsto (\theta^m_T)'(h)$ is a continuous function, therefore
\begin{equation}\label{e:grand-mere}\theta^m_T(h) =  \theta_T^m(0)+\int_0^h (\theta^m_T)'(\ell) d \ell.\end{equation}
From Lemma  \ref{l:Russo-continu-tronc} again, we can rewrite the derivative as 
$$ (\theta^m_T)'(\ell) = - \int_{ \widetilde{\mathscr{S}}^{T}} \P_{\ell} \left((x,t,u) \mbox{ is $T\stackrel{m}{-}$pivotal} \right) \un_{\Lambda_m}(x)d \widetilde{\lambda}^T(x,t,u).$$
Using the bound  \begin{equation}\label{e:queue-pivot-tronc}\P_h ((x,t,u) \mbox{ is $T\stackrel{m}{-}$pivotal}) \leq \kappa_1 \exp(-\kappa_2 ||x||_{\infty}),\end{equation} (which is proved exactly as \eqref{e:queue-pivot}) and \eqref{e:cv-ppivot-m}, 
Lebesgue's dominated convergence shows that 
$$\lim_{m \to +\infty}(\theta^m_T)'(\ell) = - \int_{ \widetilde{\mathscr{S}}^{T}} \P_{\ell} \left((x,t,u) \mbox{ is $T-$pivotal} \right) d \widetilde{\lambda}^T(x,t,u)\coloneqq g(l).$$
Thanks to \eqref{e:queue-pivot-tronc} again, $\sup_{m,\ell} |(\theta_T^m)'(\ell)|<+\infty$, 
so that, applying Lebesgue's dominated convergence theorem in \eqref{e:grand-mere} and using \eqref{e:cv-theta-m}, we deduce that 
\begin{equation}\label{e:grand-pere} \theta_T(h) = \theta_T(0)+\int_0^h g(\ell) d \ell.\end{equation}
The fact that $\P_{\ell} \left((x,t,u) \mbox{ is $T-$pivotal} \right)$ is a continuous function of $\ell$ (which is Lemma \ref{l:ppivot-continue-h}), combined with \eqref{e:queue-pivot} shows that $g$ is a continuous function, and we thus deduce from \eqref{e:grand-pere} that 
$(\theta_T)'(h) = g(h)$.
\end{proof}

\begin{lem}\label{lem_bound_derivatives}
    For any $T \geq 0$, $\sup_{h\in[0,1]}\sup_{t\in[0,T]}|\theta_t'(h)|<+\infty$. 
\end{lem}

\begin{proof}
    For any $h\in[0,1]$, $t\in[0,T]$, for any $(x,s,u)\in\widetilde{ \mathscr{S}}^t$, from \eqref{e:queue-pivotÉ} we deduce $$\P_{h} \left((x,s,u) \mbox{ is $t-$pivotal} \right) \leq \mathbb{P}_h(x\in I_0^t).$$ In addition, one can see from the construction of the sets $I_s^T$ in Section \ref{sec_influence} that $I_0^t$ has the same law as $I_{T-t}^T$ and $I_s^T$ is non-increasing with respect to $s$, hence 
    \[
    \P_{h} \left((x,s,u) \mbox{ is $t-$pivotal} \right) \leq \mathbb{P}_h(x\in I_{T-t}^T)\leq \mathbb{P}_h(x\in I_{0}^T) \leq \kappa_1\exp(-\kappa_2\|x\|_\infty)
    \]
    thanks to \eqref{e:controle-spread}. 
   This and Proposition \ref{p:Russo-continu-tronc} yield 
   \[
|\theta_t'(h)| \leq \int_{ \widetilde{\mathscr{S}}^{t}} \kappa_1\exp(-\kappa_2\|x\|_\infty) d \widetilde{\lambda}^t(x,s,u)\leq\kappa_1MT\sum_{x\in\mathbb{Z}^d}\exp(-\kappa_2\|x\|_\infty)
    \]
    which ends the proof of the lemma.
\end{proof}

\section{OSSS}\label{sec:OSSS}
\subsection{OSSS inequality for Poisson functionals}\label{sec_0SSS_defs}
We present in this section the continuous version of the OSSS inequality for Poisson functionals\footnote{In \cite{LasPecYog}, Poisson processes are viewed as random measures, while in the previous sections, we viewed Poisson processes as random sets. Both points of view are equivalent, and we abuse notation a little bit in the sequel by not taking the distinction into account.}, established by Last, Peccati and Yogeshwaran \cite{LasPecYog}. As explained in the introduction, one could use the discrete OSSS inequality instead, through a detour via a suitable discretization procedure. However, in this paper, we choose to rely on the framework developed by Last–Peccati–Yogeshwaran in order to work directly in the continuous setting.

Let $(\mathbb X,\mathcal X)$ be a Borel space. For any measure $\mu$ on $(\mathbb X,\mathcal X)$, $B\in\mathcal X$, we denote $\mu_B=\mu(B \cap \cdot)$. Let $(B_n)_{n\in\mathbb{N}}$ be an increasing sequence in $\mathcal{X}$ so that $\bigcup_{n\in\mathbb N}B_n=\mathbb X$. Let $\lambda$ be a measure on $(\mathbb X,\mathcal X)$ with $\lambda(B_n)<+\infty$ for all $n\in\mathbb N$. Let  $\eta$ be a Poisson point process on $\mathbb X$ with intensity measure $\lambda$. We denote $\mathbf{N}$ the set of all measures on $(\mathbb X,\mathcal X)$ which take integer values on all $B\in\mathcal{X}$ so that $B \subset B_n$ for some $n\in\mathbb{N}$. We endow $\mathbf N$ with the smallest $\sigma$-algebra making the maps $\mu \mapsto \mu(B)$ measurable for all $B\in\mathcal{X}$.

 We are interested in mappings $Z : \mathbf{N} \mapsto \mathcal{X}$; roughly, for $\mu\in \mathbf{N}$, $Z(\mu)$ corresponds to ``the part of $\mathbb X$ we are allowed to look at if we know $\mu$''. Such a $Z$ is called \emph{graph-measurable} if $(x,\mu) \mapsto \un_{\{x\in Z(\mu)\}}$ is a measurable mapping on $\mathbb X \times \mathbf{N}$. $Z$ is a \emph{stopping set} if it is graph-measurable and ``what we are allowed to look at depends only on the value of $\mu$ in the place we are allowed to look at'': for all $\mu,\nu\in \mathbf{N}$, $Z(\mu)=Z(\mu_{Z(\mu)}+\nu_{Z(\mu)^c})$.

To formalize the notion of a continuous-time exploration algorithm for $\mathbb X$, the following notion is used. 
A family $\{Z_t : t \in \mathbb{R}_+\}$ of stopping sets is called a \textit{continuous-time decision tree (CTDT)} if for any $t \in \mathbb{R}_+$ there exists $n\in\mathbb N$ so that $Z_t \subset B_n$, if $\mathbb{E}[\lambda(Z_0(\eta))] = 0$ and if the following properties are satisfied: $
 Z_s \subset Z_t\text{ for } s \leq t$ and $Z_t = \bigcap_{s > t} Z_s$ for $t \in \mathbb R_+$. $Z_t$ will represent the part of $\mathbb{X}$ explored at time $t$.
If $\{Z_t : t \in \mathbb{R}_+\}$ is a CTDT, we define $Z_\infty := \bigcup_{t \in \mathbb{R}_+} Z_t, Z_{t-} := \bigcup_{s < t} Z_s, \text{for } t \in \mathbb{R}_+^*$,
as well as $Z_{0-} := \emptyset$. 

Let $f : \mathbf{N} \to \mathbb{R}$ be a measurable function such that $\mathbb{E}[|f(\eta)|] < \infty$. We say that the CTDT \emph{determines} $f$ if all the information revealed during the exploration is sufficient to determine the value of $f$. More precisely, we require that $f(\mu) = f(\mu_{Z_\infty(\mu)})$ for any $\mu \in \mathbf{N}$ and we further assume that, as time $t$ increases, the information revealed up to time $t$ is sufficient to approximate the true value of 
$f$ with arbitrary precision, that is $f(\zeta_t){\stackrel{\mathbb P} {\rightarrow}} f(\eta')$ as $t \to \infty$
where $
        \zeta_t := \eta_{Z_\infty(\eta) \setminus Z_t(\eta)} + \eta'_{Z_t(\eta)} + \eta'_{\mathbb X \setminus Z_\infty(\eta)} $ with $\eta'$ being a Poisson point processs on $\mathbb X$ with intensity measure $\lambda$, independent from $\eta$.

We will also need the following conditions on the CTDT to ensure it is not degenerate and reveals at most one point at a time during the exploration
    \begin{equation}\label{cond:1OSSS}
        \lambda(Z_t(\mu) \setminus Z_{t-}(\mu)) = 0, \quad \mu \in \mathbf{N},\, t \in \mathbb{R}_+ 
        \end{equation}
        and
        \begin{equation}\label{eq:cond2OSSS}
        \mathbb{P}(\eta(Z_t(\eta) \setminus Z_{t-}(\eta)) \leq 1 \text{ for all } t \in \mathbb{R}_+) = 1.
    \end{equation}

Finally, we need to introduce an extra source of randomness in our exploration (we will randomize our starting set $Z_0$). We consider $(\mathbb Y,\mathcal Y)$ a measurable space and now allow the stopping sets to depend on values in $\mathbb Y$ in such a way that for every $y$, $\{Z_t^y, t\in\R_+\}$ is a CTDT and 
$(\mu,x,y)\mapsto \mathbf 1_{\{x\in Z_{t}^y(\mu)\}}$ is measurable on $\mathbf N\times\mathbb X\times \mathbb Y$ for all $t\in\R_+$. Let $Y$ be an independent random variable with values in $\mathbb Y$. If the above conditions are satisfied then $\{Z_t^Y, t\in\R_+\}$ is called a \emph{randomized CTDT}. We say that a randomized CTDT determines $f:\mathbf{N}\to\R$, if, for every $y\in\mathbb Y$, $\{Z_t^y, t\in\R_+\}$ determines $f$.

The following theorem corresponds to Corollary 4.1 in \cite{LasPecYog}.
\begin{thm}[Randomized OSSS inequality for Poisson functionals]\label{thm:OSSS} Let $f:\mathbf {N}\to\mathbb [-1,1]$ be measurable and let $Y$ be an independent random variable with values in $\mathbb Y$. Let $\{Z_t^Y, t\in\R_+\}$ be a randomized CTDT determining $f$ such that for $\P(Y\in\cdot)$-almost every $y\in\mathbb Y$, $\{Z_t^y\}$ satisfies \eqref{cond:1OSSS} and \eqref{eq:cond2OSSS}.
 Then,
\begin{equation}
    \mathrm{Var}(f) \leq 2 \int_{\mathbb X }\mathbb{P}(x \in Z_{\infty}^Y(\eta)) \mathbb{E}[|D_x f(\eta)|]\, \lambda(dx), 
\end{equation}
where $D_x f(\eta) := f(\eta + \delta_x) - f(\eta)$.
\end{thm}

\subsection{Bound on local influences}

The following lemma provides a comparison between the local influence appearing in the OSSS bound and the derivative appearing in Russo's formula. 
To define the CTDT and use the OSSS bound, it will be more convenient to work with the truncated model $X^m$ defined in Section \ref{sec:derivative_theta}.

\begin{lem}\label{l:borne-infl}
Let $f = \un_{ \left\{X_T^m(0) = 1 \right\}}$ when starting from $X_0^m = \mathbf{1}_{|\Lambda_m}$ ($f$ is viewed as a measurable function of $\mathscr{P}_m^T$), then let 
\begin{equation}\label{e:def-I}
I \coloneqq \int_{\mathscr{S}_m^T} \E_h \left( \left| f(\mathscr{P}_m^T \cup \{ (x,t,u,w) \}) - f(\mathscr{P}_m^T)  \right| \right) d \lambda_{h,m}^T(x,t,u,w) \end{equation}
and 
\begin{equation}\label{e:def-J}
J \coloneqq \int_{ \widetilde{\mathscr{S}}_m^T } \P_h \left((x,t,u) \mbox{ is $T\stackrel{m}{-}$pivotal} \right) d \widetilde{\lambda}_m^T(x,t,u). \end{equation}
We claim that the following bound holds:
\begin{equation}\label{e:local-infl} I \leq (M/c_1(\mathbf{1})) \cdot J.\end{equation}
\end{lem}
As mentioned in the introduction, an additional difficulty arises compared to the contact process  studied in \cite{beekenkamp}. Specifically, it is possible for the point \((x, t, u, w)\) to be pivotal (in the sense that $f(\mathscr{P}_m^T \cup \{ (x,t,u,w) \}) \neq f(\mathscr{P}_m^T)$), even though \((x, t, u)\) is not \(T\stackrel{m}{-}\)pivotal. This occurs for instance when condition (b) for being pivotal is satisfied but condition (a) is not, and \(X_t^m(x)[\mathscr{P}] = 1\). In this case, we have \(X_t^m(x)[\mathscr{P} \cup \{(x, t, u, w)\}] = 0\), meaning that the presence of the point changes the outcome. Therefore, we cannot hope for a bound on the trajectories themselves, but we may still aim for a bound involving  expectations.

\begin{proof}
Consider the event $D \coloneqq \left\{ \left| f(\mathscr{P}_m^T \cup \{ (x,t,u,w) \}) - f(\mathscr{P}_m^T)  \right| = 1 \right\}$ (the only possible non-zero value is $1$). We use the inclusion 
$D \subset  \left( \{X_{t-}^m(x) = 0 \} \cap (a)_m \cap (b)_m \right) \cup  \left( \{X_{t-}^m(x) = 1 \} \cap (b)_m \right)$,
where $(a)_m$ (resp. $(b)_m$) is the event that condition (a) (resp. (b)) for $T\stackrel{m}{-}$pivotality of $(x,t,u)$ is met.
As a consequence, $I \leq I_1 + I_2$, with
$$I_1
 \coloneqq \int_{\mathscr{S}_m^T}  \P_h \left( X_{t-}^m(x) = 0 , (a)_m \cap (b)_m \right) d \lambda_{h,m}^ T(x,t,u,w) $$
and 
$$I_2 \coloneqq\int_{\mathscr{S}_m^T}  \P_h \left( X_{t-}^m(x) = 1 , (b)_m \right) d \lambda_{h,m}^T(x,t,u,w).$$
Noting that the probability appearing in $I_1$ does not involve the value of $w$, we have that 
\begin{equation}\label{e:formule-I1} I_1
 = \int_{ \widetilde{\mathscr{S}}_m^T}  \P_h \left( X_{t-}^m(x) = 0 , (a)_m \cap (b)_m \right) d \widetilde{\lambda}_m^T(x,t,u).\end{equation}
Noting that the probability appearing in $I_2$ does not involve the value of $w$ nor of $u$, we have that 
 $$I_2
 = \int_{\Lambda_m \times [0,T]}  \P_h \left( X_{t-}^m(x) = 1 , (b)_m \right) d \widehat{\lambda}_m^T  (x,t),$$
 where $\widehat{\lambda}_m^T = M (\mbox{counting} \otimes \mbox{Lebesgue})$. Moreover, on $\{ X_{t-}^m(x) = 1 \} \cap (b)_m$, we have that $(a)_m$ is satisfied for all $u$ in the interval $[C_0,M] = [M-c_1(\mathbf{1}), M]$, so that, writing 
 $$I_2 = \frac{M}{c_1(\mathbf{1})} \int_{\Lambda_m \times [0,T] \times [M-c_1(\mathbf{1}), M]}  \P_h \left( X_{t-}^m(x) = 1 , (b)_m \right) d \widetilde{\lambda}_m^T  (x,t,u),$$
 we see that $I_2$ can be rewritten as
  $$ \frac{M}{c_1(\mathbf{1})}  \int_{\Lambda_m \times [0,T] \times [M-c_1(\mathbf{1}), M]}  \P_h \left( X_{t-}^m(x) = 1 , (a)_m \cap (b)_m \right) d \widetilde{\lambda}_m^T  (x,t,u),$$
  so that 
  \begin{equation}\label{e:formule-I2}I_2 \leq  (M/c_1(\mathbf{1}))   \int_{\widetilde{\mathscr{S}}_m^T}  \P_h \left( X_{t-}^m(x) = 1 , (a)_m \cap (b)_m \right) d \widetilde{\lambda}_m^T  (x,t,u).\end{equation}
Putting together \eqref{e:formule-I1} and  \eqref{e:formule-I2}, noting that $M/c_1(\mathbf{1}) \geq 1$, and remembering that $$ \{(x,t,u) \mbox{ is $T\stackrel{m}{-}$pivotal} \} = (a)_m \cap (b)_m,$$ we have that 
$$I_1 + I_2 \leq (M/c_1(\mathbf{1})) \int_{\widetilde{\mathscr{S}}_m^T}  \P_h \left((a)_m \cap (b)_m \right) d \widetilde{\lambda}_m^T  (x,t,u) = (M/c_1(\mathbf{1})) J.$$
\end{proof}

\subsection{Exploration process}
Let us begin with an informal description of the exploration process that is the continuous analogue of that in \cite[Lemma 2.4]{Har}. Our goal is to determine whether $X_T^m(0)=0$ when starting at $X_0^m=\mathbf{1}_{|\Lambda_m}$, using as few information on the underlying Poisson process $\mathscr{P}_m^T$ as possible. We start by sampling a random time \( S \) uniformly from the interval \([0, T]\).  We then define an auxiliary dynamics $(\check X_t^m)_{t\ge S}$ on $\{0,1\}^{\mathbb Z_m^d}$, which starts at time $S$, with configuration $\mathbf{1}_{|\Lambda_m}$, and which uses the same underlying Poisson process as $(X_t^m)_{t \geq 0}$.

 We want to know the values of the $\check X_t^m$, $S \leq t \leq T$. To do that, we start revealing the points of the Poisson process forward in time, starting from time $S$. However, if for all \( y \in x + \Lambda_R \) we have \( \check X_{t^-}^m(y) =0 \),
then \( x \) will necessarily have value \( 0 \) at time \( t \), regardless of whether there is a point at \((x, t)\) or not.  Therefore we do not reveal the Poisson process for these $(x,t)$ for now.

 By the monotonicity of the dynamics, for all $t \geq S$, $x\in\Lambda_m$, we have $X_t^m(x) \leq \check X_t^m(x)$. 
 Therefore, if \(\check X_T^{ m}(0)=0\), we know that \( X_T^{ m}(0) = 0 \), and stop the exploration process. Otherwise, if \(\check X_T^{ m}(0)=1\),  we do not know the value of $X_T^m(0)$, hence we make a similar exploration procedure starting from time 0 to determine $X_t^m$, $0 \leq t \leq T$.

More formally, we define an exploration process in the following way.  $\mathscr E_t$ will represent ``the part of the Poisson process explored before time $t$ if $S=0$'', and the $T_n$ the times at which $\mathscr{E}_t$ changes.
Start with $\mathscr{E}_{0} \coloneqq \emptyset$ and $T_0\coloneqq0$.
 Assume that, for $n \geq 0$, $T_0,\ldots,T_n$ have already been defined together with $\mathscr{E}_t$ for $t \leq T_n$, and that, for $x \in \Lambda_m$, $X_{T_n}^m(x)$ is a function of $\mathscr{P}_m^{ T} \cap \mathscr{E}_{T_n}$. 
Now let $E_{n} \coloneqq \{ x \in \Lambda_m \ | \ \exists\, y \in x + { \Lambda_R} \mbox{ s.t. }X_{T_n}^m(y) \neq 0  \}$. Define  $T_{n+1} \coloneqq \inf \{ t > T_n \ | \ \exists\, (x,t,u,w) \in \mathscr{P}_m^{\ T} \mbox{ s.t. } x \in E_{n} \}$ (which is $+\infty$ if the set is empty).  Let $\mathscr{E}_t \coloneqq \mathscr{E}_{T_n} \cup (E_n \times (T_n, t] \times [0,M] \times \{A,B\})$ for $T_n < t \leq T_{n+1}$.

Given $s \in \R$ and a subset $\Xi$ of $\Lambda_m \times [0,T] \times [0,M] \times \{A,B\}$, let 
$\varpi_s(\Xi) \coloneqq \{  (x,t-s,u,w) \ | \ (x,t,u,w) \in \Xi , \ t \geq s \}$.

Now denote by $S$ a random variable with uniform distribution over the interval $[0,T]$, independent from $\mathscr{P}_m^T$. Then for $0 \leq s \leq T-S$, let $Z_s \coloneqq \varpi_{-S}\left(\mathscr{E}_s(\varpi_S(\mathscr{P}_m^{ T}))\right)$. If $X^m_{T-S}(0) [\varpi_S(\mathscr{P}_m^{ T})] = 1$, then, for $T-S < s \leq 2T-S$, let  $Z_s \coloneqq \mathscr{E}_{s-(T-S)}(\mathscr{P}_m^{ T}) \cup Z_{T-S}$, and otherwise $Z_s \coloneqq Z_{T-S}$. Finally, for $s > 2T-S$, let    $Z_s \coloneqq Z_{2T-S}$. 

Then the family of sets $(Z_{t})_{t \geq 0}$ satisfies the assumptions of a randomized continuous-time decision tree\footnote{Strictly speaking, a CTDT involves functions that are defined for any element $\mu \in \mathbf N$, while, with our definition, $Z_{t}$ appears as a function of $\mathscr P_m^T$. It is easy to extend this definition to a function of any element $\mu \in \mathbf N$ by using the set of atoms of $\mu$ in the graphical construction instead of $\mathscr P_m^T$. To do that, we ignore the atoms of the form $(x,0,u,w)$, $(x,t,M,A)$, and if there are two atoms of the form $(x,t,u,w),(x,t,u',w')$ we ignore one of them according to an arbitrary order on $[0,M]\times\{A,B\}$.} (see Section \ref{sec_0SSS_defs}), determining the value\footnote{
To be fully precise, it determines the value of a function which outputs $X^m_T(0)$ almost surely when applied on $\mathscr P_m^T$.
} of $X_T^{m}(0)$.

We now study the probability that a point is explored. Let $(x,t,u,w)\in\Lambda_m\times [0,T]\times [0,M] \times \{A , B \}$. If $t<S$, then $(x,t,u,w)$ can be explored only if $X^m_{T-S}(0) [\varpi_S(\mathscr{P}_m^T)] = 1$. If $t> S$, $(x,t,u,w)$ is explored only if there exists $y\in x+\Lambda_R$ so that $X^m_{(t-S)^-}(y) [\varpi_S(\mathscr{P}_m^T)] = 1$. This implies 
\[
\P_h((x,t,u,w) \in Z_{\infty} )
\]
\[
\leq \frac{1}{T}\int_0^T\left(\mathbb{P}_h(X^m_{T-s}(0) [\varpi_s(\mathscr{P}_m^T)] = 1)+\un_{\{t>s\}}\sum_{y\in (x+\Lambda_R)\cap\Lambda_m}\mathbb{P}_h(X^m_{(t-s)^-}(y) [\varpi_s(\mathscr{P}_m^T)] = 1)\right)\mathrm{d}s
\]
\[
\leq \frac{1}{T}\int_0^T\left(\mathbb{P}_h(X^m_{T-s}(0) = 1)+\un_{\{t>s\}}\sum_{y\in (x+\Lambda_R)\cap\Lambda_m}\mathbb{P}_h(X^m_{(t-s)^-}(y)  = 1)\right)\mathrm{d}s.
\]
Moreover, a reasoning similar to that of Proposition \ref{prop_theta(T)} yields that $: s \mapsto \mathbb{P}_h(X^m_{s}(y)  = 1)$ is continuous, thus $\mathbb{P}_h(X^m_{(t-s)^-}(y)  = 1)=\mathbb{P}_h(X^m_{t-s}(y)  = 1)$, which implies 
\begin{equation}\label{eq:Pexplored}
\P_h((x,t,u,w) \in Z_{\infty})\leq (|\Lambda_R|+1)\frac{1}{T}\int_0^T\max_{y\in\Lambda_m}\mathbb P_h(X^m_{s}(y)=1)\mathrm{d}s.
\end{equation}
Now, for any $y\in\Lambda_m$, we can define a dynamics $(X_s^{2m,y})_{0 \leq s \leq T}$ on $(y+\Lambda_{2m}) \times [0,T]$ by using the graphical construction with the truncated Poisson process $\mathscr{P} \cap ((y+\Lambda_{2m}) \times [0,T]\times [0,M]\times\{A,B\})$ and zero boundary conditions: all sites outside $y+\Lambda_{2m}$ are assumed to be zero. If $X_0^{2m,y}=\mathbf{1}_{|\Lambda_{2m}}$, the monotonicity of the dynamics implies that for all $s\in[0,T]$, $X_s^m(y) \leq X_s^{2m,y}(y)$, hence $\mathbb{P}_h(X_s^m(y)=1) \leq \theta_s^{2m}(h)$. Consequently, \eqref{eq:Pexplored} becomes
\begin{equation}\label{eq:Pexploredbis}
\P_h((x,t,u,w) \in Z_{\infty})\leq (|\Lambda_R|+1)\frac{1}{T}\int_0^T\theta_s^{2m}(h)\mathrm{d}s.
\end{equation}

\subsection{Proof of Theorem \ref{t:BJL-1}}
Thanks to Theorem \ref{thm:OSSS} applied to the randomized CTDT $(Z_t)_{t \geq 0}$ and to $f=\mathbf 1_{\{X_T^m(0)=1\}}$,
 we have that 
$$\theta_T^m(h) (1- \theta_T^m(h))\le 2\int_{\mathscr{S}_m^T} \P_{h}(\mathbf{z} \in Z_{\infty} )  \E_{h} \left( \left| f(\mathscr{P}_m^T \cup \{ \mathbf{z} \}) - f(\mathscr{P}_m^T)  \right| \right) d \lambda_{h,m}^T(\mathbf{z}),$$
where we use the notation $\mathbf{z} = (x,t,u,w)$.
Combining \eqref{eq:Pexploredbis}, \eqref{e:local-infl} and Lemma \ref{l:Russo-continu-tronc}, we deduce that
\begin{equation}\label{e:differential-ineq-m} \theta_T^m(h) (1- \theta_T^m(h)) \leq 2 (M / c_1(\mathbf{1})) (| \Lambda_R|+1) \cdot  (-\theta^m_T)'(h)  \cdot \frac{1}{T} \int_0^T \theta_s^{2m}(h) ds.\end{equation}

We now take the limit $m \to +\infty$, using \eqref{e:cv-theta-m} (together with Lebesgue's dominated convergence theorem for the integral $\int_0^T$ term) and \eqref{e:cv-derivee}, to deduce that 
\begin{equation}\label{e:differential-ineq} \theta_T(h) (1- \theta_T(h)) \leq 2 (M / c_1(\mathbf{1})) (| \Lambda_R|+1) \cdot  (-\theta_T)'(h)  \cdot \frac{1}{T} \int_0^T \theta_s(h) ds.\end{equation}

 Thanks to Proposition \ref{prop_theta(T)}, $T\mapsto \theta_T$ is non-increasing. Moreover,  by Lemma \ref{lem: monotonicity in h}, $h \mapsto \theta_T(h)$ is non-increasing. 
Hence for $T \geq 1$, we have $\theta_T(h)\leq \theta_1(h)\leq \theta_1(0)$. Furthermore, the IPS with generator $L$ is ergodic, hence $c_0(\mathbf{1})>0$, which implies $\theta_1(0)<1$.
 Note that on the event $\{\mathscr{P}\cap(\{0\}\times [0,T]\times [0,M]\times \{A,B\})=\emptyset\}$, we have $X_T(0)=1$. It implies that $\theta_T(h)>0$ for any $T \geq 0$ and $h\in[0,1]$.  
It yields that for all $T\ge 1$ and $h\in[0,1]$
\begin{equation}\label{eq_fond} -\theta_T'(h) \ge c \frac{T}{ \int_0^T \theta_s(h) ds}\theta_T(h)\end{equation}
for some constant $c>0$ not depending on $h$ or $T$. 
Now we are able to adapt the analysis argument from \cite{OSSSdiscret} in our continuous setting. Note that compared to its original version with a discrete sequence of functions, it further requires a uniform upper bound in $h$ of the derivatives.

\begin{lem}\label{lem:analyse} Let $c>0$.
    Consider a family of non-increasing differentiable functions \( f_T : [0, h_0 ] \to (0, 1] \), $T\ge 0$ satisfying that for any $h\in[0,h_0]$, the limit $f(h) \coloneqq \lim_{T \to \infty} f_T(h)$  exists, the function $T\mapsto f_T(h)$ is continuous and 
$-f_T'(h) \geq c \frac T{\Sigma_T(h)} f_T(h)$ for all $T\ge 1,\,h\in[0,h_0]$
where \(\Sigma_T(h)\coloneqq\int_{0}^T f_t(h)dt \). Assume also that for all $T\ge 1$, 
\begin{equation}\label{cond:sup}
  \sup_{h\in[0,h_0]}  \sup_{ t\in[0,T]}|f'_t(h)| < +\infty.
\end{equation}
Then, there exists \( h_1 \in [0, h_0] \) such that:

\begin{itemize}
    \item For any \( h > h_1 \), there exists \( c_h > 0 \) such that for any \( T \) large enough, $f_T(h) \leq \exp(-c_h T)$.
    \item  For any \( h < h_1 \),  we have $f(h) \geq c(h_1-h)$.
\end{itemize}

\end{lem}

Before proving this lemma, let us first conclude the proof of Theorem \ref{t:BJL-1}. We consider the family of functions $\theta_T : [0,1] \to (0,1]$. They are non-increasing 
by Lemma \ref{lem: monotonicity in h}, and differentiable thanks to Proposition \ref{p:Russo-continu-tronc}. Proposition \ref{prop_theta(T)} implies $T \mapsto \theta_T(h)$ is continuous and non-increasing and hence the limit $\lim_{T\rightarrow\infty }\theta_T(h)$ exists for every $h\in[0,1]$. Furthermore, we have \eqref{eq_fond}, and thanks to Lemma \ref{lem_bound_derivatives}, the condition \eqref{cond:sup} is satisfied, hence Lemma \ref{lem:analyse} applies.   
Since the IPS corresponding to $h=0$ is ergodic, we have
$\lim_{T\rightarrow\infty} \theta_T(0)=0$.
This implies that $h_1=0$ and for all $h>0$, there exists \( c_h > 0 \) such that for any \( T \) large enough, $\theta_T(h) \leq \exp(-c_h T)$.
This proves exponential ergodicity for the IPS with generator $L_h$. Moreover, if we set $h = \frac{\varepsilon}{M+\varepsilon}$, the IPS associated to $L_h$ with time rescaled by  a factor $\frac{1}{1-h}$ has the same law as the IPS associated with $L_\varepsilon$, hence the IPS associated with $L_\varepsilon$ has exponential ergodicity.
This concludes the proof of Theorem \ref{t:BJL-1}. 

\begin{proof}[Proof of Lemma \ref{lem:analyse}]
Define $h_1 := \sup \{ h : \limsup_{T \to \infty} \frac{\log \Sigma_T(h)}{\log T} \geq 1 \}$. We use the convention that $h_1=0$ if the set of such $h$ is empty.

 Assume \( h >h_1 \). 
 We will prove that there is exponential decay of \( f_T(h) \) in two steps. We set $h'\in(h_1,h)$, $h''=\frac{h+h'}{2}$ and $\delta=h-h''$. First,  since $h'>h_1$ there exist \( T_0>1\) and \( \alpha > 0 \) such that $\Sigma_T(h') \leq T^{1-\alpha}$ for all $T \geq T_0$. Since the $f_t$ are non-increasing, $\Sigma_T(\tilde h) \leq T^{1-\alpha}$ holds for all $\tilde h \in [h',h_0]$, $T \geq T_0$. For such a \( T \), integrating \( -f_T' \geq c T^\alpha f_T \) between \( h' \) and \( h'' \) implies that for all $T \geq T_0$, $f_T(h'') \leq f_T(h') \exp(-c \delta T^\alpha)\le \exp(-c \delta T^\alpha)$. 
 
 This implies that there exists \( \Sigma < \infty \) such that \( \Sigma_T(h'') \leq \Sigma \) for all \( T \). Integrating \( -f_T' \geq \frac c \Sigma Tf_T \) between \( h'' \) and \( h \) gives $f_T(h) \leq \exp\left(-\frac c\Sigma\delta T\right)$ for all $T \geq 1$.

 { We now study the case $h<h_1$. }For \( T {>} 1 \), define the function \( F_T (h):= \frac{1}{\log T} \int_{1}^T \frac {f_t(h)}{t}dt \). 
Note that $F_T(h)-f(h)= \frac 1 {\log T}\int_1 ^T \frac {f_t(h)-f(h)}{t}dt$
where we recall that $f(h)=\lim_{T\rightarrow\infty}f_T(h)$.
It follows easily from the previous equality that $F_T(h)$ converges to $f(h)$ as $T$ goes to infinity. Thanks to \eqref{cond:sup}, we can differentiate $F_T$ in $h$, we obtain 
\[
        F'_T (h)= \frac{1}{\log T} \int_1 ^T \frac {f_t'(h)}{t}dt\le -\frac{c}{\log T} \int_1 ^T \frac {f_t(h)}{\Sigma_t(h)}dt = -\frac c{\log T} (\log\Sigma_{T}(h)-\log\Sigma_{1}(h)) 
\]
where we used in the last equality that $\frac {\partial}{\partial T}\Sigma_T(h)=f_T(h)$ recalling that $T\mapsto f_T(h)$ is continuous. For \( h' \in (h, h_1) \), using that \( h\mapsto \Sigma_T(h)\) is non-increasing and integrating the previous differential inequality between \( h \) and \( h' \) gives $F_T(h) - F_T(h') \geq c(h'-h)\frac{\log \Sigma_T(h')-\log \Sigma_1(h)}{\log T}$. Hence, the fact that \( F_T(h) \) converges to \( f(h) \) as \( T \to \infty \) implies
\[
f(h) - f(h') \geq c(h' - h) \limsup_{T \to \infty} \frac{\log \Sigma_T(h)-\log \Sigma_1(h')}{\log T}\geq c(h' - h).
\]
Letting \( h' \) tend to \( h_1 \) from below, we obtain $f(h) \geq c(h_1 - h)$. The result follows.
\end{proof}

\section*{Acknowledgements}

The authors wish to thank Ivailo Hartarsky for pointing out the interest of the question studied in this paper and providing us with various references.

\bibliographystyle{plain}
\bibliography{OSSS-IPS}

\end{document}